\documentclass[11pt,twoside,a4paper]{amsart}
\usepackage[french,english]{babel}
\usepackage{amssymb,bbm,url,color}
\usepackage{amsmath}
\usepackage{amsthm}
\usepackage[dvips]{graphics}
\usepackage{graphicx}
\usepackage{latexsym}
\usepackage[normal]{subfigure}
\input{epsf.sty}
\include{psfig}
\usepackage[all]{xy}
\usepackage{eufrak}
\usepackage{mathrsfs}
\usepackage{lipsum}
\usepackage{mdframed}
\usepackage{perpage}
\usepackage{amsfonts}
\usepackage{fancyhdr}
\usepackage{enumerate,fancyhdr}
\usepackage{epsfig,graphicx}
\PassOptionsToPackage{dvips}{graphicx}
\PassOptionsToPackage{pdftex}{graphicx,color,url}
\usepackage{ifpdf}
\ifpdf
\usepackage{hyperref}
\usepackage[para]{footmisc}
\usepackage{subfloat}
\begin{document}
\def\K{\mathbb{K}}
\def\R{\mathbb{R}}
\def\C{\mathbb{C}}
\def\Z{\mathbb{Z}}
\def\Q{\mathbb{Q}}
\def\D{\mathbb{D}}
\def\N{\mathbb{N}}
\def\T{\mathbb{T}}
\def\P{\mathbb{P}}
\def\A{\mathscr{A}}
\def\CC{\mathscr{C}}
\renewcommand{\theequation}{\thesection.\arabic{equation}}
\newtheorem{theorem}{Theorem}[section]
\newtheorem{cond}{C}
\newtheorem{lemma}{Lemma}[section]
\newtheorem{corollary}{Corollary}[section]
\newtheorem{prop}{Proposition}[section]
\newtheorem{definition}{Definition}[section]
\newtheorem{remark}{Remark}[section]
\newtheorem{example}{Example}[section]
\bibliographystyle{plain}

\title[Riemann-Poisson and K\"ahler-Poisson complex manifolds]{Riemann-Poisson and K\"ahler-Poisson complex manifolds and structures}
\author[I.\  Hamidine \& M.\ Saminou\ Ali ]{Ibrahima Hamidine \& Mahamane Saminou Ali}

\address{D\'epartement de Math\'ematiques et Informatique\\ Facult\'e des Sciences et Techniques\\ Universit\'e Abdou Moumouni, BP: 10662 Niamey (Niger)}

\email{ih5877@gamil.com}
\address{D\'epartement de Math\'ematiques\\ Facult\'e des Sciences et Techniques\\ Universit\'e d'Agadez, BP: 199 Agadez (Niger)}
\email{mahamanesaminou@gmail.com}
\thanks{This works was supported by the EDP-MC Network with finance support from the International Science Program (ISP Grants Burk01).}
\keywords{Contravariant connections, Poisson manifolds and structure, complex structure, Hermitian form.}

\subjclass{32Q15, 32Q60, 53C15, 53D17, 37J15}

\date{\today}

\maketitle
\begin{abstract}In this paper we consider structures of complex  Poisson brackets on the space of smooth functions in a $n$-dimensional complex manifold generated by the $(1,1)$-form $d=\partial+\overline{\partial}$-closed and non-degenerate (with non-holomorphic and non-antiholomorphic coefficients). In this case, we have view the compatibility between complex Poisson and Riemannian structures and an example is giving in $\C^\ast$.
\end{abstract}

\section{Introduction}
Poisson geometry is considered to be the transition between non-commutative algebra and differential geometry. From the point of view of physicians, it is the transition between classical mechanic and quantum mechanic, because Poisson geometry play a crucial role in the development of quantization. Weinstein (\cite{We})  showed its essential importance in Hamiltonian mechanic.

A Poisson structure on a differential manifold $M$ is the data of a Lie bracket on  the space of smooth functions $\mathscr{C}^{\infty}(M)$ satisfying the Jacobi's identity. It's also equivalently defined as the data of  a bi-vector field $\pi$ satisfying some condition. Since, the introduction of the notion of contravariant connection by Vaismann in \cite{IV}, and its generalization by Fernandes in \cite{Fe}, the Poisson geometry has evolved  considerably.    

The study of compatibility between  Poisson and Riemann structures was began with Vaisman \cite{IV}, and it was generalized by Boucetta \cite{MB1,MB2,MB3}. It's motivated by the geometry behind a symplectic form and give a  lot of properties of the non degenerate symplectic leaves \cite{Ha, AHB}.  

Lot of notions of compatibilities between Poisson and Riemann structures were done in literature, all generalizing K\"ahler manifold. In spite of, there is not a theory of manifolds unifying Poisson and Riemann structures. Alba and Vargas proposed in \cite{AV} a notion of compatibility between Poisson, metric and partial complex structures, conjointly with the geometric conditions of integrabilities.

In complex geometry, Poisson tensor were generally defined as a $(2,0)$-vector field on a complex  holomorphic space on manifold $M$ \cite{AV}. In the special case of a toric complex manifold, the Poisson tensor were defined as a $(1,1)$-vector field by \cite{AC,CG}.

In this paper, we study the compatibility between complex Riemann, K\"ahler, Poisson structures. We consider complex  Poisson brackets on the space of smooth functions on a $n$-dimensional complex  manifold generated by the $(1,1)$-form $d$-closed and non-degenerate (with non-holomorphic and non-antiholomorphic coefficients).

This paper is organized as follow. After this introduction giving, the second section give some preliminaries. The third section, give the Poisson tensor as a $(1,1)$-vector field and some of its  properties. The forth section study the notion of compatibility between the three structures. In the last section, an example is giving.    
  
\section{Preliminary}
\subsection{Complex manifolds}
Let $M$ be a $n$-dimensional analytic complex manifold and $U$ a domain of $M$. A holomorphic vector field on $U$ is a holomorphic map $X : U \subset M \rightarrow T^{1,0}M,\; p \mapsto X(p) \in T^{1,0}M$.
\vskip 1mm
\noindent
In a local coordinates $\{z_1, \dots, z_n\}$ of $U$, one has:
\begin{eqnarray*}
X= \sum_{j=1}^n X_j \frac{\partial}{\partial z_j},
\end{eqnarray*}
where $X_j: U \rightarrow \mathbb{C}$ are holomorphic functions.

\begin{definition}{\rm 
Let $M$ be a $n$-dimensional analytic complex manifold. The holomorphic tangent bundle is a $CR$-structure on $M$ if the following assertions are satisfied:
\begin{enumerate}
\item $T^{1,0}_zM \cap T^{0,1}_zM = \{0\}$, for all $z\in M$;
\item For all $Z,W \in \mathscr{C}^{\infty}(U, T^{1,0}M)$, then $[Z,W]\in \mathscr{C}^{\infty}(U, T^{1,0}M)$, for $U \subset M$ a domain of $M$. Here $T^{0,1}M = \overline{T^{1,0}M}$ is the conjugate of $T^{1,0}M$.
\end{enumerate}  
}
\end{definition}

\begin{prop}
Let $M$ be a $n$-dimensional analytic complex manifold and $X$ and $Y$ two hohomorphic vectors fields on a domain $U \subset M$. Then, the Lie bracket of $X$ and $Y$ denoted by $[X,Y]:= XY -YX$ is a holomorphic vector field on $U$. 
\end{prop}

 \begin{proof}
 If we identify $M$ to $\C^n$, then for all $z\in \C^n$, we have locally:
 \begin{eqnarray*}
 X(z) = \sum_{j=1}^n X_j(z)\frac{\partial }{\partial z_j} \quad \mbox{and} \quad Y(z) = \sum_{k=1}^n Y_k(z) \frac{\partial}{\partial z_k}.
 \end{eqnarray*}
 For a holomorphic function $f \in \mathscr{C}^\infty(U)$, one has:
 \begin{eqnarray*}
 [X,Y]f &=& (XY)f - (YX)f \\
        &=& \sum_{j=1}^nX_j(z)\frac{\partial}{\partial z_j}\Big(\sum_{k=1}^nY_k(z)\frac{\partial}{\partial z_k}f\Big)-\sum_{k=1}^nY_k(z)\frac{\partial}{\partial z_k}\Big(\sum_{j=1}^nX_j(z)\frac{\partial}{\partial z_j}f\Big)\\
&=&\sum_{j,k}^nX_j(z)\frac{\partial f}{\partial z_k}\frac{\partial Y_k(z)}{\partial z_j}-\sum_{j,k}^nY_k(z)\frac{\partial f}{\partial z_j}\frac{\partial X_j(z)}{\partial z_k}\\
&=&\sum_{k=1}^n\Big(\sum_{j=1}^nX_j(z)\frac{\partial Y_k(z)}{\partial z_j}-Y_j(z)\frac{\partial X_k(z)}{\partial z_j}\Big)\frac{\partial f}{\partial z_k}.
 \end{eqnarray*}
 Then, 
 \begin{eqnarray*}
 [X,Y] = \sum_{k=1}^n\Big(\sum_{j=1}^nX_j(z)\frac{\partial Y_k(z)}{\partial z_j}-Y_j(z)\frac{\partial X_k(z)}{\partial z_j}\Big)\frac{\partial }{\partial z_k}.
 \end{eqnarray*}
 \end{proof}
 
 \begin{corollary}
 If $X$ and $Y$ are two holomorphic vectors fields, then 
 \begin{eqnarray*}
 [X,Y] = -[Y,X],
 \end{eqnarray*}
 and  the Lie operators $L_X$ and $L_Y$ satisfy:
 \begin{eqnarray*}
 L_{[X,Y]} = L_X\circ L_Y - L_Y \circ L_X.
 \end{eqnarray*}
 \end{corollary}

\subsection{Pseudo-Riemannian Poisson Manifold}
Let $(M,g,\pi)$ be a $n$-dimensional manifold endowed with a pseudo-Riemannian metric $g$ and a Poisson structure $\pi$.  In \cite{IV}, the author introduce a notion of compatibility of Poisson structure and Riemannian metric by the condition $\nabla \pi = 0$, where $\nabla$ is the Levi-Civita connection associated to $g$. This condition induce that the Poisson structure is regular. On the other hand, both of Poisson structures are not regular. Hence, this condition is not universal. The author of \cite{MB1,MB2,MB3} introduce the notion of pseudo-Riemann Poisson manifold and give a notion of compatibility which is more general than the above one.

Let $(M,\pi)$ be a Poisson manifold. A pseudo-Riemannian  $g$ induce on $T^\ast M$ a pseudo-Riemannian metric $g^\ast$ such that:
$$g^\ast (\alpha, \beta) = g (\sharp_g \alpha, \sharp_g \beta), \qquad \forall \alpha,\beta \in \Omega^1(M),$$
where $\Omega^1(M)$ is the space of $1$-form on $M$ and $\sharp_g:T^\ast M \to TM$ is a musical isomorphism associated to $g$. Then we have the following theorem \cite{MB1}.

\begin{theorem}
Let  $(M,\pi,g)$ be a manifold endowed with a Poisson tensor $\pi$ and a pseudo-Riemannian metric $g$. There is a unique contravariant connection $D$, on $M$ such that for all $\alpha,\beta \in \Omega^1(M)$:
\begin{enumerate}
\item[1)] $\pi_\# (\alpha)\cdot g^\ast(\beta,\gamma) = g^\ast ({D}_\alpha \beta,\gamma) + g^\ast (\beta,{D}_\alpha \beta)$,
\item[2)] $[\alpha,\beta]_\pi = {D}_\alpha\beta-{D}_\beta \alpha$.
\end{enumerate}  
\end{theorem}
$D$ is called the Levi-Civita contravariant connection associated to $g$. It's also characterized by the Koszul's formula:
\begin{eqnarray}\label{ED1}
2g^\ast ({D}_\alpha\beta,\gamma)& =& \pi_\# (\alpha)\cdot g^\ast (\beta,\gamma) + \pi_\# (\beta)\cdot g^\ast (\alpha,\gamma)-\pi_\# (\gamma)\cdot g ^\ast (\alpha,\beta)
+g^\ast ([\alpha,\beta]_\pi ,\gamma) \nonumber \\
& &\qquad + g^\ast([\gamma,\alpha]_\pi ,\beta) + g^\ast([\gamma,\beta]_\pi ,\alpha).
\end{eqnarray}

Hence, $(M,\pi,g)$ is said to be a pseudo-Riemann Poisson manifold if :
\begin{eqnarray*}
{D}_\alpha \pi(\beta,\gamma) := \pi_\# (\alpha)\cdot\pi(\beta,\gamma)-\pi({D}_\alpha \beta,\gamma)-\pi(\beta,{D}_\alpha \gamma) = 0.
\end{eqnarray*}

\section{Complex Poisson Manifold}
Let $M$ be a  $n$-dimensional differential complex manifold.

\begin{definition}{\rm 
A $(1,1)$-vector field $X$ on $M$ is a section of the fiber bundle $T^{1,1}M$. A $(1,1)$-form on $M$ is a dual of a $(1,1)$-vector field.}  
\end{definition}

In a local coordinates $(z_1, \dots, z_n)$, a $(1,1)$-vector field $X$ is given by:
\begin{eqnarray*}
X= \sum_{j,k=1}^n X_{jk}\frac{\partial}{\partial z_j}\wedge \frac{\partial}{\partial \overline{z}_k},
\end{eqnarray*}
where $X_{jk}$ are smooth functions in a domain $U$ of $M$.
\vskip 1mm
\noindent
Let $U$ be a domain of $M$. We denote by $\Omega^{1,1}(U)$ the space of the differentials $(1,1)$-forms on $U$.
A $(1,1)$-form $\alpha\in \Omega^{1,1}(U)$, can be writ in this local coordinates as:
\begin{eqnarray*}
\alpha = \sum_{j,k=1}^n \alpha_{jk}dz_j \wedge d\overline{z}_k,
\end{eqnarray*}
where $\alpha_{jk}$ are smooth functions on $U$ and its differential is given by:
\begin{eqnarray*}
d\alpha = \sum_{j,k=1}^n d \alpha_{jk} \wedge dz_j \wedge d\overline{z}_k,
\end{eqnarray*}
with $d= \partial + \overline{\partial}$ such that $d^2 = \partial^2 = \overline{\partial}^2 = 0$ and $\partial\overline{\partial} + \overline{\partial}\partial = 0$.

\begin{remark}{\rm 
The following $(1,1)$-form 
\begin{eqnarray*}
\frac{1}{(2i\pi)^2}\bigwedge\limits_{j=1}^n dz_j \wedge dz_j  = \bigwedge\limits_{k=1}^n(dx_k \wedge d y_k),
\end{eqnarray*}
is a volume form.}
\end{remark}
Let $\pi$ be a $(1,1)$-vector field on $M$, giving by:
\begin{equation}\label{EQP1}
\pi=\sum_{j,k=1}^n\pi_{jk}\partial z_j\wedge\partial\overline{z}_k, \quad \mbox{ with } \quad \partial z_j = \frac{\partial}{\partial z_j}.
\end{equation}
\begin{theorem}
The bracket $\{f,g\} = \pi (df,dg)$ define a Poisson structure on $M$, i.e
\begin{enumerate}
\item[1)] $\{f,g\}=-\{g,f\}$, the skew symmetry;
\item[2)] $\{fg,h\}=f\{g,h\}+\{f,h\}g$, Leibniz's rule;
\item[3)] $\{f,\{g,h\}\}+\{g,\{h,f\}\}+\{h,\{f,g\}\}=0\;\Longleftrightarrow\;[\pi,\pi]_\pi=\frac{1}{2}A_J=0$, the  Jacobi's identity.
\end{enumerate}
\end{theorem}

\begin{proof}
\begin{enumerate}
\item[1)] 
\begin{eqnarray*}
\{f,g\}&=&\pi(df,dg)\nonumber\\
&=& \sum_{j,k=1}^n\pi_{jk}\Big(\partial z_jf\partial\overline{z}_kg-\partial\overline{z}_jf\partial z_kg\Big)
\end{eqnarray*}
\begin{eqnarray*}
\{g,f\}&=&\pi(dg,df)\nonumber\\
&=&-\sum_{j,k=1}^n\pi_{jk}\Big(\partial z_jf\partial\overline{z}_kg-\partial\overline{z}_jf\partial z_kg\Big)\nonumber\\
&=&-\{f,g\}.
\end{eqnarray*}
\item[2)] \begin{eqnarray*}
\{fg,h\}&=&\sum_{j,k}^n\pi_{jk}\Big(f\partial z_jg\partial\overline{z}_kh+g\partial z_jf\partial\overline{z}_kh-f\partial\overline{z}_jg\partial z_kh-g\partial\overline{z}_jf\partial z_kh\Big)\nonumber\\
&=&\sum_{j,k=1}^nf\pi_{jk}\Big(\partial z_jg\partial\overline{z}_kh-\partial\overline{z}_jg\partial z_kh\Big)+\sum_{j,k=1}^ng\pi_{jk}\Big(\partial z_jf\partial\overline{z}_kh-\partial\overline{z}_jf\partial z_kh\Big)\nonumber\\
&=&f\sum_{j,k=1}^n\pi_{jk}\Big(\partial z_jg\partial\overline{z}_kh-\partial\overline{z}_jg\partial z_kh\Big)+g\sum_{j,k=1}^n\pi_{jk}\Big(\partial z_jf\partial\overline{z}_kh-\partial\overline{z}_jf\partial z_kh\Big)\nonumber\\
&=&f\{g,h\}+g\{f,h\},
\end{eqnarray*}
hence the Leibniz's rule.
\item[3)] A straightforward calculus, give:
\begin{eqnarray}\label{EQJac1}
\{f,\{g,h\}&=&\sum_{j,k=1}^n\sum_{p,q}^n\pi_{jk}\partial\overline{z}_k\pi_{pq}\partial z_jf\partial z_pg\partial\overline{z}_qh-\sum_{j,k=1}^n\sum_{p,q}^n\pi_{jk}\partial\overline{z}_k\pi_{pq}\partial z_jf\partial \overline{z}_pg\partial z_qh\nonumber\\
& &-\sum_{j,k=1}^n\sum_{p,q=1}^n\pi_{jk}\partial z_k\pi_{pq}\partial\overline{z}_jf\partial z_pg\partial\overline{z}_qh+\sum_{j,k=1}^n\sum_{p,q=1}^n\pi_{jk}\partial z_k\pi_{pq}\partial\overline{z}_jf\partial \overline{z}_pg\partial z_qh\nonumber\\
& &+\sum_{j,k=1}^n\sum_{p,q=1}^n\pi_{jk}\pi_{pq}\partial z_jf\partial^2_{\overline{z}_kz_p}g\partial\overline{z}_qh-\sum_{j,k=1}^n\sum_{p,q=1}^n\pi_{jk}\pi_{pq}\partial z_jf\partial_{\overline{z}_p}g\partial^2_{z_k\overline{z}_q}h\nonumber\\
&&-\sum_{j,k}^n\sum_{p,q=1}^n\pi_{jk}\pi_{pq}\partial \overline{z}_jf\partial z_pg\partial^2_{\overline{z}_kz_q}h+\sum_{j,k}^n\sum_{p,q=1}^n\pi_{jk}\pi_{pq}\partial\overline{z} _jf\partial^2_{z_k\overline{z}_p}g\partial_{z_q}h,
\end{eqnarray}
\begin{eqnarray}\label{EQJac2}
\{g,\{h,f\}\}&=&\sum_{j,k=1}^n\sum_{p,q=1}^n\pi_{jk}\partial\overline{z}_k\pi_{pq}\partial z_jg\partial z_ph\partial\overline{z}_qf-\sum_{j,k=1}^n\sum_{p,q=1}^n\pi_{jk}\partial\overline{z}_k\pi_{pq}\partial z_jg\partial \overline{z}_ph\partial z_qf\nonumber\\
& &-\sum_{j,k=1}^n\sum_{p,q=1}^n\pi_{jk}\partial z_k\pi_{pq}\partial\overline{z}_jg\partial z_ph\partial\overline{z}_qf+\sum_{j,k=1}^n\sum_{p,q=1}^n\pi_{jk}\partial z_k\pi_{pq}\partial \overline{z}_jg\partial\overline{z}_ph\partial z_qf\nonumber\\
& &+\sum_{j,k=1}^n\sum_{p,q=1}^n\pi_{jk}\pi_{pq}\partial z_jg\partial^2_{\overline{z}_kz_p}h\partial\overline{z}_qf-\sum_{j,k=1}^n\sum_{p,q=1}^n\pi_{jk}\pi_{pq}\partial z_jg\partial\overline{z}_ph\partial^2_{z_k\overline{z}_q}f\nonumber\\
& &-\sum_{j,k=1}^n\sum_{p,q=1}^n\pi_{jk}\pi_{pq}\partial \overline{z}_jg\partial z_ph\partial^2_{\overline{z}_kz_q}f+\sum_{j,k=1}^n\sum_{p,q=1}^n\pi_{jk}\pi_{pq}\partial\overline{z}_jg\partial_{z_k\overline{z}_p}h\partial{z}_qf,
\end{eqnarray}
and
\begin{eqnarray*}
\{h,\{f,g\}\}&=&\sum_{j,k=1}^n\sum_{p,q=1}^n\pi_{jk}\partial\overline{z}_k\pi_{pq}\partial z_j h\partial z_pf\partial\overline{z}_qg-\sum_{j,k=1}^n\sum_{p,q=1}^n\pi_{jk}\partial\overline{z}_k\pi_{pq}\partial z_j h\partial \overline{z}_pf\partial z_qg\nonumber\\
& &-\sum_{j,k=1}^n\sum_{p,q=1}^n\pi_{jk}\partial z_k\pi_{pq}\partial\overline{z}_jh\partial z_pf\partial\overline{z}_qg+\sum_{j,k=1}^n\sum_{p,q=1}^n\pi_{jk}\partial z_k\pi_{pq}\partial \overline{z}_jh\partial \overline{z}_pf\partial z_qg\nonumber\\
& &+\sum_{j,k=1}^n\sum_{p,q=1}^n\pi_{jk}\pi_{pq}\partial z_jh\partial^2\overline{z}_kz_p f\partial\overline{z}_qg-\sum_{j,k=1}^n\sum_{p,q=1}^n\pi_{jk}\pi_{pq}\partial z_jh\partial\overline{ z}_pf\partial^2z_k\overline{z}_qg\nonumber
\end{eqnarray*}
\begin{eqnarray}\label{EQJac3}
& &-\sum_{j,k=1}^n\sum_{p,q=1}^n\pi_{jk}\pi_{pq}\partial\overline{z}_jh\partial z_pf\partial^2\overline{z}_kz_qg+\sum_{j,k=1}^n\sum_{p,q=1}^n\pi_{jk}\pi_{pq}\partial\overline{z}_jh\partial^2z_k\overline{z}_pf\partial z_qg.
\end{eqnarray}
By combining the equations  \eqref{EQJac1}, \eqref{EQJac2} and \eqref{EQJac3}, one get
\begin{eqnarray}
\{f,\{g,h\}\}+\{g,\{h,f\}\}+\{h,\{f,g\}\}&=&0,
\end{eqnarray}
hence one have the Jacobi's identity.
\end{enumerate}
\end{proof}

With $\pi$ such that the functions $\pi_{jk}$ define an Hermitian matrix, then we have the following proposition.

\begin{prop}
If the functions $\pi_{jk}$ define an Hermitian matrix, then the above Poisson tensor is real. 
\end{prop}

\begin{proof}
Let $\pi = \sum\limits_{j,k=1}^n\pi_{jk}\partial z_j \wedge \partial \overline{z}_k$ be a $(1,1)$-vector field, where $\partial z = \frac{\partial}{\partial z}$ and $\overline{\pi}_{jk} = - \pi_{kj}$. Note that $[\pi_{jk}(z)]_{1\leq j,k\leq n}$ is an Hermitian matrix depending analytically of $z$. Hence, one has:
\begin{eqnarray*}
\overline{\pi}&=&\overline{\sum\limits_{j,k=1}^n\pi_{jk}\partial z_j\wedge\partial\overline{z}_k}=\sum\limits_{j,k=1}^n\overline{\pi}_{jk}\partial \overline{z}_j\wedge\partial z_k\\
&=&\sum\limits_{j,k=1}^n\pi_{kj}\partial z_k\wedge\partial\overline{z}_j\\
&=&\pi.
\end{eqnarray*}
\end{proof}

\begin{example}{\rm 
Let $U$ be an open subset of $\C^n$. Then, 
\begin{eqnarray*}
\pi=2i\sum_{j,k=1}^nz_j\overline{z}_k\partial z_j\wedge\partial\overline{z}_k
\end{eqnarray*}
is a $(1,1)$-real Poisson tensor.
\vskip 1mm
Indeed,
\begin{eqnarray*}
\overline{\pi}&=&\overline{2i\sum_{j,k=1}^nz_j\overline{z}_k\partial z_j\wedge\partial\overline{z}_k}=-2i\sum_{j,k=1}^n\overline{z}_jz_k\partial \overline{z}_j\wedge\partial z_k\\
&=&2i\sum_{j,k=1}^nz_k\overline{z_j}\partial z_k\wedge\partial\overline{z}_j\\
&=&\pi.
\end{eqnarray*}
} 
\end{example}

A vector field $X$ on a complex Poisson manifold $(M,\pi)$, is said to be a Poisson vector field if it satisfies the following condition:
\begin{eqnarray}
[X,\pi] = 0.
\end{eqnarray}

\begin{theorem}
Let $(M,\pi)$ be a complex Poisson manifold endowed with a $(1,1)$-vector field $\pi= \sum\limits_{j,k=1}^n\pi_{jk}\partial z_j \wedge \partial \overline{z}_k $. Let $X= \sum\limits_{j=1}^n X_j\partial z_j + \sum\limits_{k=1}^n\overline{X}_k\partial \overline{z}_k$ be a vector field on $TM$ such that the function $X_j$ are holomorphic and $\overline{X}_k$ antiholomorphic. Then, $X$ is a Poisson vector field on $(M,\pi)$, if and only if the functions $X_j$ and $\overline{X}_k$  satisfy the following equation :
\begin{eqnarray}\label{XP1}
\partial z_j\Big(\frac{X_j}{\pi_{jk}}\Big)+\partial\overline{z}_k\Big(\frac{\overline{X_k}}{\pi_{jk}}\Big)=0,\; \forall \, j,k=1,\dots, n.
\end{eqnarray}
\end{theorem}

\begin{proof}
Suppose $X'_j=X_j\partial z_j$, the we have:
\begin{eqnarray*}
[X'_j,\pi_{jk}\partial z_j\wedge\partial\overline{z}_k]&=&X'_j(\pi_{jk})\partial z_j\wedge\partial\overline{z}_k+\pi_{jk}[X'_j,\partial z_j]\wedge\partial\overline{z}_k+\pi_{jk}\partial z_j\wedge[X'_j,\partial\overline{z}_k].
\end{eqnarray*}
Let $f$ a smooth function on $M$.
\begin{eqnarray*}
[X'_j,\partial z_j](f)=[X_j\partial z_j,\partial z_j](f)=-\partial z_j(X_r)\partial z_r(f),
\end{eqnarray*}
and
\begin{eqnarray*}
[X'_j,\partial \overline{z}_k](f)=[X_j\partial z_j,\partial \overline{z}_k](f)=-\partial\overline{z}_k(X_j)\partial z_j(f)=0.
\end{eqnarray*}
Hence
\begin{eqnarray}\label{XP2}
[X'_j,\pi_{jk}\partial z_j\wedge\partial\overline{z}_k]&=&X'_j(\pi_{jk})\partial z_j\wedge\partial\overline{z}_k+\pi_{jk}\Big(-\partial z_j(X_j)\partial z_j\wedge \partial \overline{z}_k\Big).
\end{eqnarray}
Likewise, with $\overline{X}'_k = \overline{X}_k\partial\overline{z}_k$, one has:
\begin{eqnarray}\label{XP3}
[\overline{X'_k},\pi_{jk}\partial z_j\wedge\partial\overline{z}_k]&=&\overline{X'_k}(\pi_{jk})\partial z_j\wedge\partial\overline{z}_k+\pi_{jk}\Big(-\partial \overline{z}_k(\overline{X_k})\partial z_j\wedge \partial \overline{z}_k\Big).
\end{eqnarray}
By adding \eqref{XP2} and \eqref{XP3}, and  exchanging $j$ and $k$ we have
\begin{eqnarray*}
[X,\pi]&=&\sum_{j,k=1}^n\Big( X_j\partial z_j(\pi_{jk})\partial z_j\wedge\partial\overline{z}_k-\pi_{jk}\partial z_j(X_j)\partial z_j\wedge \partial \overline{z}_k\Big)\nonumber\\
& &+\sum_{k,j=1}^n\Big(\overline{X_k}\partial\overline{z}_k(\pi_{jk})\partial z_j\wedge\partial\overline{z}_k-\pi_{jk}\partial \overline{z}_k(\overline{X_k})\partial z_j\wedge \partial \overline{z}_k\Big).\nonumber\\
&=&-\sum_{j,k=1}^n\pi_{jk}^2\Big[\partial z_j\Big(\frac{X_j}{\pi_{jk}}\Big)+\partial\overline{z}_k\Big(\frac{\overline{X_k}}{\pi_{jk}}\Big)\Big]\partial z_j\wedge \partial \overline{z}_k.
\end{eqnarray*}
Which proof the theorem.
\end{proof}
As a consequence of this theorem, the following corollary give a characterization of holomorphic and antiholomoprhic Poisson vector field. 
\begin{corollary}
With the hypothesis as in the above theorem, an holomorphic vector field $X = \sum\limits_{j=1}^nX_j\partial z_j$ (respectively antiholomorphic vector field $X = \sum\limits_{k=1}^n\overline{X}_k\partial \overline{z}_k$) is a Poisson vector field, if:
\begin{eqnarray}
X_j = f(\overline{z}_j)\pi_{jk}\quad (\mbox{respectively}\; \overline{X}_k = g(z_k)\pi_{jk}),
\end{eqnarray}
where $f$ is an antiholomorphic function (respectively $g$ holomorphic).
\end{corollary}
Let $M$ be an analytic complex manifold and $J:TM \to TM$ an almost complex structure such that $J^2 = - {\rm Id}$. For $X,Y \in T_zM$, one has the integrability's conditions:
\begin{enumerate}
\item[1)] $[JX,Y]+[X,JY]=0$.
\item[2)] $J\Big([JX,Y]+[X,JY]\Big)=[JX,JY]-[X,Y]$.
\end{enumerate}
Note that this conditions are equivalent to the vanishing of the Nihenjuis bracket:
$$N(X,Y]=[JX,JY]-[X,Y]-J[JX,Y]-J[X,JY]=0.$$
Using the conventional equality $z_j=x_j+iy_j$, $x,y\in\R$, one has:
\begin{eqnarray*}
\frac{\partial}{\partial z_j}&=&\frac{1}{2}\Big(\frac{\partial}{\partial x_j}-i\frac{\partial}{\partial y_j}\Big),\qquad  \frac{\partial}{\partial \overline{z}_j}\,=\,\frac{1}{2}\Big(\frac{\partial}{\partial x_j}+i\frac{\partial}{\partial y_j}\Big)
\end{eqnarray*}
and $J\Big(\frac{\partial}{\partial x_j}\Big))=\frac{\partial }{\partial y_j}$, $J\Big(\frac{\partial}{\partial y_j}\Big)=-\frac{\partial}{\partial x_j}$ (like a retrograde  rotation of $\frac{\pi}{2}$). Hence we have:
\begin{eqnarray*}
J\Big(\frac{\partial}{\partial z_j}\Big)&=&\frac{1}{2}\Big(\frac{\partial}{\partial y_j}+i\frac{\partial}{\partial x_j}\Big)=\frac{i}{2}\Big(\frac{\partial}{\partial x_j}-i\frac{\partial}{\partial y_j}\Big)=i\frac{\partial}{\partial z_j}\\
J\Big(\frac{\partial}{\partial \overline{z}_j}\Big)&=&\frac{1}{2}\Big(\frac{\partial}{\partial y_j}-i\frac{\partial}{\partial x_j}\Big)=-\frac{i}{2}\Big(\frac{\partial}{\partial x_j}+i\frac{\partial}{\partial y_j}\Big)=-i\frac{\partial}{\partial\overline{z}_j}.
\end{eqnarray*}

\begin{prop}
Let $(M,\pi)$ be a complex Poisson manifold. Then $\pi$ is $J$-invariant.
\end{prop}

\begin{proof}
\begin{eqnarray*}
J\pi&=&J\Big(\sum\limits_{j,k=1}^n\pi_{jk}\partial z_j\wedge\partial\overline{z}_k\Big)=\sum\limits_{j,k=1}^n\pi_{jk}J\Big(\partial z_j\wedge\partial\overline{z}_k\Big)\\
&=&\sum\limits_{j,k=1}^n\pi_{jk}J(\partial z_j)\wedge J(\partial\overline{z}_k)=\sum\limits_{j,k=1}^n\pi_{jk}(i\partial z_j)\wedge(-i\partial\overline{z}_k)\\
&=&\sum\limits_{j,k=1}^n\pi_{jk}\partial z_j\wedge\partial\overline{z}_k\\
&=&\pi.
\end{eqnarray*}
\end{proof}

\begin{remark}
The Poisson bracket of two complex functions $f$ and $g$ on $M$ is giving by:
$$\{f ,g\} = \pi(df ,dg).$$
In local coordinates, one has:
$$\{f ,g\} =\sum_{j,k=1}^n\pi_{jk}\Big(\frac{\partial f}{\partial z_j}\frac{\partial g}{\partial\overline{z}_k}-\frac{\partial f}{\partial\overline{ z}_j}\frac{\partial g}{\partial z_k}\Big).$$
If $f$ and $g$ are both holomophic or both antiholomorphic, then $\{f ,g\}=0$. 
\vskip 1mm
\noindent
Recall that the vector  field $X_f$ of $f$ is giving by:
$$X_f (g) = \{f ,g\}$$ 
is called Hamiltonian vector field and in local coordinates it can be expressed by:
$$X_f=\sum_{k=1}^n\Big[\sum_{j=1}^n\pi_{jk}\Big(\frac{\partial f}{\partial z_j}\frac{\partial }{\partial\overline{z}_k}-\frac{\partial  f}{\partial\overline{ z}_j}\frac{\partial }{\partial z_k}\Big)\Big].$$
If $f$ is an holomorphic (respectively antiholomorphic) function, then we have:
$$X_f=\sum_{k=1}^n\Big(\sum_{j=1}^n\pi_{jk}\frac{\partial f}{\partial z_j}\frac{\partial }{\partial\overline{z}_k}\Big ) \quad (\mbox{respectively} \quad X_f=\sum_{j=1}^n\Big(\sum_{k=1}^n\overline{\pi}_{kj}\frac{\partial  f}{\partial\overline{ z}_j}\frac{\partial }{\partial z_k}\Big)).$$
\end{remark}

\section{Compatibility of Poisson and Riemann complex structures }
Let $h$ be a Hermitian metric on a complex manifold $M$ locally define by:
\begin{eqnarray}
h=\sum\limits_{j,k=1}^nh_{jk}dz_j\otimes d\overline{z}_k,
\end{eqnarray}
where the functions $h_{jk}$ defined a Hermitian matrix and:
$$dz\otimes d\overline{z}_k(\partial z_j,\partial z_k)=dz_j(\partial z_j)\cdot d\overline{z}_k(\partial\overline{z}_k)=\delta_{jk};$$
$$dz\otimes d\overline{z}_k(\partial z_j,\partial \overline{z}_k)=dz_j(\partial z_j)\cdot d\overline{z}_k(\partial z_k)=0.$$

\begin{lemma}\label{LemH}
$h$ is $J$-invariant.
\end{lemma}

\begin{proof}
For $X,Y\in TM$, i.e
$$X=\sum_{j=1}^n(X_j\partial z_j+\overline{X}_j\partial\overline{z}_j),\qquad\,Y=\sum_{k=1}^n(Y_k\partial z_k+\overline{Y}_k\partial\overline{z}_k),$$
one has
$$h(X,Y)=\sum_{j,k=1}^n(X_j\overline{Y}_k+\overline{X}_jY_k).$$
Other where, on has $$JX=i\sum_{j=1}^n\Big(X_j\partial z_j-\overline{X}_j\partial\overline{z}_j\Big),\qquad\,JY=i\sum_{k=1}^n\Big(Y_k\partial z_k-\overline{Y}_k\partial\overline{z}_k\Big).$$
Hence,
\begin{eqnarray}\label{EqH1}
h(JX,JY)&=&\sum_{j,k=1}^nh(iX_j\partial z_j-i\overline{X}_j\partial\overline{z}_j,iY_k\partial z_k-i\overline{Y}_k\partial\overline{z}_k)\nonumber\\
&=&-i^2\sum_{j,k=1}^nh(X_j\partial z_j-\overline{X}_j\partial\overline{z}_j,Y_k\partial z_k-\overline{Y}_k\partial\overline{z}_k)\nonumber\\
&=&\sum_{j,k=1}^n\Big(h(X_j \partial z_j,Y_k\partial z_k)+h(\overline{X}_j\partial\overline{z}_j,\overline{Y}_k\partial\overline{z}_k)\Big)
\end{eqnarray}
because $h(\partial z_k,\partial \overline{z}_k)=0$.
\vskip 1mm
Likewise,
\begin{eqnarray}\label{EqH2}
h(X,Y)&=&\sum_{j,k=1}^nh(X_j\partial z_j+\overline{X}_j\partial\overline{z}_j,Y_k\partial z_k+\overline{Y}_k\partial\overline{z}_k)\nonumber\\
&=&\sum_{j,k=1}^n\Big(h(X_j \partial z_j,Y_k\partial z_k)+h(\overline{X}_j\partial\overline{z}_j,\overline{Y}_k\partial\overline{z}_k)\Big).
\end{eqnarray}
By \eqref{EqH1} and \eqref{EqH2}, we have $h(JX,JY)=h(X,Y)$.
\end{proof} 
 
Form the Lemma \ref{LemH}, $h$ can be writ as the some of a Hermitian form $h_1$ on $T^{1,0}M$ and a Hermitian form $h_2$ on $T^{0,1}M$, i.e.
$$h =h_1+h_2.$$

This Hermitian form induced a Riemannian metric $g$, on $TM$ by:
  $$g(X,Y) = \frac{1}{2}\Big( h + \overline{h}\Big)(X,Y), \qquad \forall \, X,Y \in TM;$$
  and a symplectic form $\omega$ on $TM$ by:
  $$\omega(X,Y) = g(X,JY), \qquad \forall \, X,Y \in TM.$$
  Locally, the symplectic $(1,1)$-form $\omega$ can be expressed as:
$$\omega = \frac{i}{2}\sum_{j,k=1}^nh_{jk}dz_j\wedge d\overline{z}_k.$$  

\begin{remark}
The Hermitian form $h$ can be decomposed as:
$$h= g + i\omega.$$
\end{remark}
 The following result is well know \cite[Chapter 5]{D.all}:
\begin{prop}\label{Prop}
Let $M$ be a complex manifold endowed with a symplectic form $\omega$ and a Riemannian metric $g$. If $J$ is integrable, then $(M,g,\omega, J)$ is a complex Riemann-K\"ahler manifold. 
\end{prop} 

\begin{proof}
If we denote $\sharp_g$ and $\sharp_\omega$ the musicals isomorphism induced respectively by the metric $g$ and the symplectic form $\omega$, then the triple $(g,\omega, J)$ is said to be compatible if and only if, any one of the three structures can be specified by:
\begin{eqnarray*}
g (X, Y) &=& \omega (X, JY)\\
\omega( X , Y) &=& g ( JX , Y)\\
J (X) &=& (\sharp_ g )^{-1} \Big( \sharp _\omega( X )\Big).
\end{eqnarray*}
Hence  $\omega$ and $J$ are compatible if and only if $\omega(\cdot, J\cdot)$ is a Riemannian metric. Using some properties of the symplectic form $\omega$, one can stat that an almost complatible structure $J$ is an almost K\"ahler structure for the Riemannian metric  $\omega(\cdot, J\cdot)$.
\end{proof}
  
\begin{theorem}
Any Hermitian manifold is a complex Riemann-K\"ahler manifold.
\end{theorem}
\begin{proof}
It's just an application of the  Lemme \ref{LemH} and above Proposition \ref{Prop}.
\end{proof}

\begin{definition}
Let $(M,g,\pi)$ be a complex manifold endowed with a Riemannian metric $g$ and a Poisson tensor $\pi$. $(M,g,\pi)$ is said to be a complex Riemann-Poisson manifold, if 
\begin{eqnarray}\label{ED2}
{D}_\alpha \pi(\beta,\gamma) := \pi_\# (\alpha)\cdot\pi(\beta,\gamma)-\pi({D}_\alpha \beta,\gamma)-\pi(\beta,{D}_\alpha \gamma) = 0,
\end{eqnarray}
where $\alpha,\beta,\gamma \in \Omega^1(M)$.
\end{definition}
\begin{remark}[\cite{AHB}]\label{rem}
Since $D$ is torsion free, on can obtain
\begin{eqnarray}
0=-[\pi,\pi](\alpha,\beta,\gamma)=D\pi(\alpha,\beta,\gamma)+D\pi(\beta,\gamma,\alpha)+D\pi(\gamma,\alpha,\beta)\\
D\pi(\gamma,\alpha,\beta)=-d\gamma\Big(\pi_\sharp(\alpha),\pi_\sharp(\beta)\Big)-\pi(D_\alpha\gamma,\beta)-\pi(\alpha,D_\beta\gamma)\\
\pi(D_\alpha\beta)-\pi(D_\beta\alpha)=[\pi_\sharp(\alpha),\pi_\sharp(\beta)].
\end{eqnarray}
\end{remark}
The following proposition is an analogous result of characterizing of pseudo-Riemannian-Poisson manifold in \cite[proposition 2.1]{MB3}. 
\begin{prop}
Let $(M, \pi, g^\ast)$ be a complex Poisson manifold endowed with a Riemannian metric on $T^*M$. Let $D$ be  the Levi-Civta contravaiant connection associated with the couple $(\pi, g^\ast)$. Then the following assertions are equivalent.
\begin{enumerate}
\item The triplet $(M, \pi, g^\ast)$ is a complex Riemannian-Poisson manifold.
\item For all $\alpha,\beta \in \Omega^1(M)$ and all $f\in \mathscr{C}^\infty(M)$, 
\begin{eqnarray}
\pi(D_\alpha df, \beta) + \pi(\alpha, D_\beta df) =0.
\end{eqnarray}
\item For all $\alpha, \beta, \gamma \in \Omega^1(M)$,
\begin{eqnarray}
d\gamma (\pi(\alpha),\pi( \beta)) + \pi(D_\alpha\gamma, \beta,) + \pi(\alpha,D_\beta\gamma )= 0. 
\end{eqnarray}
\end{enumerate} 
\end{prop}

\begin{proof}$~$
\begin{itemize}
\item[$1)\, \Rightarrow \, 2)$]
$$D\pi = 0\, \Leftrightarrow \,\pi_\sharp(\alpha) \pi(\beta, \gamma) = \pi(D_\alpha \beta, \gamma) + \pi(\beta, D_\alpha \gamma).$$
Hence
\begin{eqnarray*}
\pi(D_\alpha df, \beta) &=& \pi_\sharp (\alpha)\pi(df,\beta) - \pi(df, D_\alpha \beta) \\
\pi(\alpha, D_\beta df) &=& \pi_\sharp(\beta) \pi(\alpha, df) - \pi(D_\beta \alpha,df). 
\end{eqnarray*}
By combining  the above equations, one has:
\begin{eqnarray*}
\pi(D_\alpha df, \beta) + \pi(\alpha, D_\beta df) 
= \pi_\sharp(\alpha) \beta(\pi_\sharp(df)) - \pi_\sharp(\beta)\alpha(\pi_\sharp(df)) - [\alpha,\beta]_\pi(\pi_\sharp(df)).
\end{eqnarray*}
From \cite[lemma 2.2]{AV}, we have   
$$[\alpha, \beta]_\pi(\pi_\sharp(df)) = \pi_\sharp(\alpha)\beta(\pi_\sharp(df)) - \pi_\sharp(\beta)\alpha(\pi_\sharp(df) + \mathcal{L}_{\pi_\sharp(df)}\pi(\alpha,\beta). $$
Therefore
\begin{eqnarray*}
\pi(D_\alpha df, \beta) + \pi(\alpha, D_\beta df) = - \mathcal{L}_{\pi_\sharp(df)}\pi(\alpha,\beta).
\end{eqnarray*}
From Cartan's formula, one can deduced
\begin{eqnarray*}
\mathcal{L}_{\pi_\sharp(df)}\pi(\alpha,\beta) = d(\pi_\sharp(df))(\pi(\alpha,\beta) + d(\pi(\alpha,\beta) (d(\pi_\sharp(df)))) = 0
\end{eqnarray*}
because of $$d(\pi_\sharp(df) )= d\Big( \sum_{j,k}\pi_{jk}(\frac{\partial f}{\partial z_j}\partial\overline{z}_k - \frac{\partial f}{\partial \overline{z}_j}\partial z_k)\Big) =0.$$
\item[$2)\, \Rightarrow \, 3)$] By choosing $\gamma = df$, one can conclude.
\item[$3)\, \Rightarrow \, 1)$] It's come from the remark \ref{rem}. 
\end{itemize}
\end{proof}
\begin{theorem}
Let $(M,g)$ be a complex manifold endowed with a Riemannian metric $g$. Let $\pi$ be a Poisson tensor on $M$. Then the following assertions are equivalent:
\begin{enumerate}
\item $(M,\pi,g)$ is a complex Riemann-Poisson manifold.
\item $(M,\pi,g,J)$ is a K\"ahler-Poisson manifold.
\end{enumerate}
\end{theorem}
\begin{proof}
It's an application of \cite[Proposition 3.7]{AV}.
\end{proof}
\section{Example}
Let $\C^\ast$ be a complex manifold endowed with a Hermitian metric denoted $g$ define by:
$$g=z\overline{z} dz\otimes d\overline{z}.$$
The induced dual metric on $\C^\ast$, denoted $g^\ast$ is giving by:
\begin{eqnarray*}
g^\ast(dz,dz)=z\overline{z} &=&g^\ast(d\overline{z},d\overline{z}) \\
g^\ast(d\overline{z},dz)&=&g^\ast(dz,d\overline{z})=0.
\end{eqnarray*} 
Let $\pi$ be a Poisson tensor defined by: 
\begin{equation}
\pi=\pi_{11}\partial z\wedge\partial\overline{z}.
\end{equation}

\begin{prop}
$(\C^\ast,g, \pi)$ is a complex Riemann-Poisson (Hermitian-Poisson) manifold if and only if $\pi_{11} = c|z|^2$, $c\in \C^*$.
\end{prop}

\begin{proof}
The Levi-Civita contravariant connection of $g$ is defined by its Christoffel, which the non vanishing are:
\begin{eqnarray*}
\Gamma_{11}^1&=&\frac{\pi_{11}}{2\overline{z}}\\
\Gamma_{11}^2&=&\frac{\pi_{11}}{2\overline{z}}-\frac{\partial\pi_{11}}{\partial z}\\
\Gamma_{12}^1&=&\frac{-\pi_{11}}{2 z}+\frac{\partial\pi_{11}}{\partial z}\\
\Gamma_{12}^2&=&\frac{\pi_{11}}{2 \overline{z}}\\
\Gamma_{22}^1&=&\frac{-\pi_{11}}{2\overline{z}}+\frac{\partial\pi_{11}}{\partial\overline{z}}\\
\Gamma_{22}^2&=&\frac{-\pi_{11}}{2 z}\\
\Gamma_{21}^1&=&-\frac{\partial\pi_{11}}{\partial z}+\Gamma_{12}^1=-\frac{\pi_{11}}{2z}\\
\Gamma_{21}^2&=&-\frac{\partial\pi_{11}}{\partial\overline{z}}+\Gamma_{12}^2=-\frac{\partial\pi_{11}}{\partial\overline{z}}+\frac{\pi_{11}}{2 \overline{z}}.
\end{eqnarray*}
Using the equation \eqref{ED2}, one has
\begin{eqnarray*}
\left\lbrace\begin{array}{ccc}
D_{dz}\pi(dz,dz)&=&0 \\
D_{dz}\pi(d\overline{z},dz)&=&-\pi_{11}\frac{\partial\pi_{11}}{\partial\overline{z}}+\frac{\pi_{11}^2}{\overline{z}} \\
D_{dz}\pi(dz,d\overline{z})&=& \pi_{11}\frac{\partial\pi_{11}}{\partial\overline{z}}-\frac{\pi_{11}^2}{\overline{z}}  \\
D_{d\overline{z}}\pi(dz,dz)&=& 0\\
D_{d\overline{z}}\pi(dz,d\overline{z})&=&\frac{\pi_{11}^2}{z}-\pi_{11}\frac{\partial\pi_{11}}{\partial z} \\
D_{d\overline{z}}\pi(d\overline{z},dz)&=&-\frac{\pi_{11}^2}{z}+\pi_{11}\frac{\partial\pi_{11}}{\partial z} \\
D_{d\overline{z}}\pi(d\overline{z},d\overline{z})&=& 0 .
\end{array}
\right.
\end{eqnarray*}
Hence   $D\pi = 0$ if and only if:
\begin{eqnarray}\label{EQPISOLV}
\left\lbrace\begin{array}{ccc}
\frac{\pi_{11}^2}{z}-\pi_{11}\frac{\partial\pi_{11}}{\partial z}&=&0\\
\frac{\pi_{11}^2}{\overline{z}}-\pi_{11}\frac{\partial \pi_{11}}{\partial\overline{z}}&=&0
\end{array}\right.\;\Longleftrightarrow\; \frac{\pi_{11}^2}{z}-\frac{1}{2}\frac{\partial}{\partial z}( \pi_{11}^2)=0.
\end{eqnarray}  
The last equation of the equivalence \eqref{EQPISOLV}, has as solution the functions $f(z) = c|z|^2$, for a constant $c\in\C$.
\end{proof}
As example, if we choose $c=2i$, then $\pi = 2i|z|^2 \partial z\wedge \partial\overline{z}$, and the triplet  $(\C^\ast,g, \pi)$ is a Hermitian-Poisson (or complex Riemann-Poisson) manifold. In this case $\pi$ is real.

If  $X=a(z)\partial z+b(z)\partial\overline{z}$ is a vector field on $\C^\ast$, then one has:
$$[X,\partial z]=-\partial z(a(z))\partial z,$$
and
$$[X,\partial \overline{z}]=-\partial \overline{z}(b(z))\partial \overline{z}.$$
Hence,
\begin{eqnarray*}
[X,\pi]&=&-c^2|z|^4\Big[\partial z\Big(\frac{a(z)}{cz\overline{z}}\Big)+\partial\overline{z}\Big(\frac{b(z)}{cz\overline{z}}\Big)\Big]\partial z\wedge\partial\overline{z}.
\end{eqnarray*}
$$[X,\pi]=0\;\Longleftrightarrow\; \partial z\Big(\frac{a(z)}{cz\overline{z}}\Big)+\partial\overline{z}\Big(\frac{b(z)}{cz\overline{z}}\Big)=0.$$

If $X$ is holomorphic (i.e. $b=0$), then $a= kz$, with $k\in \C$.

If $X$ is antiholomorphic (i.e. $a=0$), then $b=k\overline{z}$, with $k\in \C$.


\begin{thebibliography}{99}
\bibitem[AHB]{AHB} M. S. Ali, M. Hassirou and M. Bazanfare, Splitting of lightlike submanifolds of pseudo-Riemannian manifolds, Glob. J. Adv. Res. Clas. Mod. Geom., 6, (2017), (2), 62-71. 
\bibitem[MB1]{MB1} M. Boucetta, Compatibilit\'e des structures pseudo-riemanniennes et des structures de Poisson, C. R. Acad. Sci. Paris, t 336,S\'erie I, 763-768 (2001).
\bibitem[MB2]{MB2} M. Boucetta, Poisson manifolds with compatible pseudo-metric and pseudo-Riemannian Lie algebras, Differiential Geometry and its Applications, Vol. 20, Issue 3, 279-291 (2004).
\bibitem[MB3]{MB3} M. Boucetta, Riemann Poisson manifold and K\"ahler Riemann foliations, C. R. Math., Acad. Sci. Paris, t 336, No.5, 423-428 (2003). Zbl 1042.53014
\bibitem[AC]{AC} Arlo CAINE, Toric Poisson structures. \url{arXiv:0910.0229v1}.
\bibitem[CG]{CG} Arlo CAINE and Berit Nilsen GIVENS, On Toric Poisson Structures of Type $(1,1)$
and their Cohomology, Symmetry. Integrability and Geometry: Methods and Applications. SIGMA 13 (2017), 023, 16 pages.
\bibitem[AV]{AV} N. M. ALBA and A. VARGAS. On the geometry of compatible Poisson and Riemannian structures. \url{arXiv:1709.02525v1}. 
\bibitem[D.all]{D.all} S. Dragomir et al. (eds), Geometry of Chauchy-Riamann Submanifolds, Springer Sciences+Bisiness Media Singapore 2016, \url{DOI 10.1007/978-981-10-0916-$7_5$}.
\bibitem[Fe]{Fe} R. L. Fernandes, Connections in Poisson geometry I. Holonomy and invariants, J. Diff. Geom., 54, no. 2, 303-365(2000).
\bibitem[Ha]{Ha} M. Hassirou, Introduction in Poisson manifolds with lightlike kaehler foliation, Pionneer J. Math. Math Sci., 7, (2013), 37-50.
\bibitem[IV]{IV} I. Vaisman, Lecture on the geometry of Poisson manifolds, Prog. In Math., Vol. 118, Berlin, Birkhausser, 1994.
\bibitem[We]{We} A. Weinstein, The local structure of Poisson manifolds, J. Diff. Geom. 18 (1993), 523-557.
\end{thebibliography}
\end{document}